\documentclass[final, 12pt]{article}
\usepackage[english]{babel}
\usepackage[latin1]{inputenc}
\usepackage{hyperref}
\usepackage{amssymb,amsmath,amsthm}
\usepackage{enumerate}
\usepackage[conditional,light,first,bottomafter]{draftcopy}
\draftcopyName{DRAFT\space\today}{130}
\draftcopySetScale{65}
\usepackage[letterpaper, hmargin=3.3cm,vmargin={2.5cm,3cm}]{geometry}
\usepackage{setspace}
\makeatletter
\renewcommand{\section}{\@startsection%
{section}%
{1}%
{0em}%
{1.7em}%
{1.2em}%
{\normalfont\large\centering\bfseries}}
\renewcommand{\@seccntformat}[1]%
{\csname the#1\endcsname.\hspace{0.5em}}
\makeatother

\numberwithin{equation}{section}
\usepackage{fancyhdr}
\usepackage[usenames]{color}

\numberwithin{equation}{section}
\newtheorem{theorem}{Theorem}[section]

\newtheorem{lemma}{Lemma}[section]

\theoremstyle{definition}
\newtheorem{definition}{Definition}[section]

\theoremstyle{remark}
\newtheorem{remark}{Remark}[section]

\def\ocirc#1{\ifmmode\setbox0=\hbox{$#1$}\dimen0=\ht0 \advance\dimen0
  by1pt\rlap{\hbox to\wd0{\hss\raise\dimen0
  \hbox{\hskip.2em$\scriptscriptstyle\circ$}\hss}}#1\else {\accent"17 #1}\fi}
\DeclareMathOperator{\re}{Re}
\DeclareMathOperator{\im}{Im}
\DeclareMathOperator{\dom}{dom}

\DeclareMathOperator{\divergence}{div}
\begin{document}
\begin{titlepage}
\title{Spectra of the Gurtin-Pipkin type equations with the kernel, depending on the parameter} 


\author{
\textbf{Romeo Perez Ortiz\footnote{Supported by the Mexican Center for Economic and Social Studies (CEMEES, by its Spanish acronym).}, Victor V. Vlasov\footnote{Supported by the Russian Foundation for Basic Research, project N14-01-00349a and N13-01-00384a}}
\\[6mm]
\small Faculty of Mechanics and Mathematics \\[-1.6mm]
\small Moscow Lomonosov State University \\[-1.6mm]
\small Vorobievi Gori, Moscow, 119991, Russia\\[1mm]
\small\texttt{cemees.romeo@gmail.com}\\[-1mm]
\small\texttt{vlasovvv@mech.math.msu.su}
}
\date{}
\maketitle
\vspace{4mm}
\begin{center}
\begin{minipage}{5in}
  \centerline{{\bf Abstract}} \bigskip In this paper, we study asymptotic behavior of the spectrum of the abstract Gurtin-Pipkin integro-differential equation with the kernel, depending on the parameter. The coefficients of this equation are unbounded and the main part is an abstract hyperbolic equation perturbed by terms that include Volterra integral operators.   
 \end{minipage}
\end{center}
\thispagestyle{empty}
\end{titlepage}
\section{Introduction}\label{introduction}
We study functional differential and integro-differential equations with unbounded operator coefficients in a Hilbert space. The main part 
of the equation under consideration is an abstract hyperbolic-type equation, disturbed by terms involving Volterra operators. These equations can be regarded as an abstract form of the Gurtin-Pipkin equation that describes heat transfer in materials with memory or sound propagation in viscoelastic media. In \cite{GMJ}, for instance, there are countless examples about this equation.   

Let $A$ be a self-adjoint positive operator with domain $\dom (A)\subset H$, where $H$ is a Hilbert space. We introduce a class of second-order abstract models 
\begin{align}
\frac{d^2u}{dt^2}+A^mu+k u- &\int_{0}^{t}K(t-s)A^{2\xi} u(s) ds=f(t), \hspace{0.3cm} t \in \mathbb{R}_{+}, \hspace{0.3cm} m =1, 2 \label{sytem1.1}\\
&u(+0)=\varphi_0, \hspace{0.4cm} u^{(1)}(+0)=\varphi_1 \label{sytem1.2},   
\end{align} where $\varphi_0$ and  $\varphi_1$ will be described later. Here, $\xi$ is a real number in the range $[0,1]$, $k$ is a non-negative constant and $K$ is the kernel associated to system  (\ref{sytem1.1}). 

The system (\ref{sytem1.1})-(\ref{sytem1.2}) with $K \in C^1\cap L_1(\mathbb{R}_+)$ represents an isotropic viscoelastic model if 
$\xi=1/2$, $k=0$, $m=1$   and $Au=-\mu \Delta u-(\lambda+\mu)\triangledown(\divergence u)$ or $Au=-\Delta u$ where $\Delta$ is the Laplacian and $\triangledown$ is the gradient of a tensor field $u$. Likewise, the system  represents a model of ionized atmosphere  
if $\xi=0$, $k >0$, $m=1$  and $A=-\Delta$.  

In \cite{JMF} the authors showed that the solutions for system (\ref{sytem1.1})-(\ref{sytem1.2}) with $m=1$, $\xi \in (0, 1/2)$ and $k=0$  decays polinomially as $t\to +\infty$, even if the kernel $K$ decays exponentially. They also showed that the dissipation given by the memory effect is not strong enough to produce exponential stability of the system (\ref{sytem1.1})-(\ref{sytem1.2})  when $\xi \in [0, 1/2)$, $k=0$ and $m=1$. That is, the corresponding semigroup associated to this problem does not decay exponentially when $\xi \in [0, 1/2)$, but such dissipation is capable to produce polinomial decay under appropriate norms. While in \cite{MFBL}, Fabrizio and Lazzari,  assuming that $\xi=1/2$, $k=0$ and kernel decays exponentially, obtained the exponential decay  of the system (\ref{sytem1.1})-(\ref{sytem1.2}).  

The equation (\ref{sytem1.1}) is called the Gurtin-Pipkin equation.  
When we consider spectral properties of
boundary-value problems for this specific equation we can  establish that properties of
the spectrum essentially depend on the kernel smoothness $K(t)$ at 
$t = 0$.

It is noteworthy that the Gurtin-Pipkin equation appears  in various domains such as mechanics and  physics as heat theory, theory
of viscoelastic media, and kinetic theory of gases. In heat theory and theory of viscoelastic media, the
kernel $K(t)$ is determined through experiments. In \cite{P}, for instance, the properties of heat conduction with memory were
studied. There, a smooth function was considered as a kernel. In \cite{PG}, the properties of solutions
were studied under the assumption that $K(t)$ is smooth, and in \cite{PI}, solutions of control problems with compact support
by boundary control and distributed control were studied. 

In theory of viscoelasticity, a kernel $K(t)$ is determined through experiments as well. The obtained curves
are often approximated by a finite sum of exponents:
\begin{align}
  K(t) \approx \sum_{k=1}^{N} c_k \exp(-\gamma_k t).
\end{align}
The dynamics of one-dimensional viscoelastic medium are described by the following second-order equation with respect t:
\begin{align}\label{equation-pipkin-gurtina}
\rho \frac{d^2u}{dt^2}=k u_{xx}+\beta \frac{du_{xx}}{dt} +\int_{0}^{t}K(t-s)u_{xx}(s) ds.
\end{align} The latter equation, integrating with respect to $t$, can be transformed to an equation that is analogous
to the Gurtin-Pipkin equation with an additional summand $\beta \frac{du_{xx}}{dt}$, which  corresponds to an
instant Kelvin-Voight friction in an original model. If $\beta = 0$ then
the equation (\ref{equation-pipkin-gurtina}) will fully
correspond to the Gurtin-Pipkin equation with the following
convolution kernel:
\begin{align}
  K(t)= \sum_{k=1}^{N} c'_{k} \exp(-\gamma_k t)+c_0.
\end{align} 

The Gurtin-Pipkin equation also appears in the theory of strongly nonhomogeneous media, in particular, in the averaging procedure of a two-phased medium containing two liquids. It is assumed that the mixture has a
periodic structure (model case) and the linear size of a cell of periodicity is equal to $\epsilon$, where $\epsilon$ is a
small parameter. Every cell of periodicity consists of two phases, namely, a liquid of the first type
and a liquid of the second type. Both liquids are incompressible; the velocities have the form $\epsilon^2 \mu_i$ $(i = 1, 2)$, where $\mu_1$ and $\mu_2$ are some constants corresponding to each liquid. The {\em passage to the limit}
as $\epsilon$ tends to $0$ in boundary-value problems for two-phased medium was considered in \cite{ESP} 
(see also \cite{VVZ1}). There, the autors assumed that the linearized equations for a viscous compressible liquid hold for each liquid phase, and the natural conditions of continuity for displacements and tensions hold on the interface.
The  {\em passage to the limit} as $\epsilon$ tends to $0$ for the specified boundary value problem for two-phased media
leads to the equation for maximum sound pressure. This equation has the form
\begin{align}
\frac{dp}{dt}=\int_{0}^{t}\divergence D(t)\triangledown_x p(\overline{x}, s) ds,  \hspace{0.2cm} \overline{x} \in \mathbb{R}^{3},
\end{align}
 where $D(t)$ is the dynamical matrix, whose coefficients $d_{ij} (t)$ are time functions
\begin{align}
  d_{ij} (t)= \sum_{k=1}^{\infty} c_{k}^{ij} \exp(-\gamma_k t)+c^{ij}_0.
\end{align}
If the inclusion of one phase into another within a cell of periodicity has the full symmetry (i.e., it is
symmetric with respect of three mutually perpendicular planes of symmetry in the tree-dimensional
case), then $D(t)$ is a scalar matrix and is defined by a diagonal element $d(t)$, for which the
following representation takes place:
\begin{align}
  d (t)= \sum_{k=1}^{\infty} c_{k} \exp(-\gamma_k t)+c_0.
\end{align}
If we consider only homogeneous motions of an effective medium (all known functions depend
on one space variable), then the equation for sound pressure will have the form of the Gurtin-Pipkin
equation and 
\begin{align}
  K(t)= \sum_{k=1}^{\infty} c_{k} \exp(-\gamma_k t)+c_0,
\end{align} where
\begin{align}\label{condition-suma}
  \sum_{k=1}^{\infty} c_{k}<\infty, \hspace{0.2cm} \sum_{k=1}^{\infty} \gamma_k c_{k}=\infty,
\end{align} which is essential for the analysis of eigen-oscillations of the medium considered.
The condition (\ref{condition-suma}) implies that the value $K(0)$ is finite, but the derivative $K'(t)$ has a singularity at $t = 0$. These conditions can be proved rigorously by methods of averaging theory (see,
e.g., \cite{ESP}). In \cite{ESP}, the terms of the sequence $\{-\gamma_k \}$ are points of the spectrum of a special Stokes-type
problem with periodic conditions, and the terms of the sequence $\{c_k \}$ are coefficients in decomposition
of the discontinuous function $r(x)$  such that, on the phase space, $r(x)=\rho_{1}^{-1}$ correspond  to the first liquid and $r(x)=\rho_{2}^{-1}$  to the second liquid  
with respect to eigen-functions of this problem where $\rho_{1}$ and $\rho_{2}$ are densities of the liquids. Note that the function $r(x)$ is discontinuous on the hypersurface, then the series  $\sum_{k=1}^{\infty} \gamma_k c_{k}$ is
divergent. Otherwise, the function $r(x)$ would belong to the space $H^1=W^{1, 2}$, which contradicts the
presence of discontinuities on the hypersurface. If the microstructure of the mixture of liquids has
continuous density in the space $H^1=W^{1, 2}$ (i.e., if the mixture does not have sharp boundaries between
phases), then the following conditions hold:
\begin{align}
  \sum_{k=1}^{\infty} c_{k}<\infty, \hspace{0.2cm} \sum_{k=1}^{\infty} \gamma_k c_{k}< \infty
\end{align}
This corresponds to the finiteness of $K(0)$ and $K'(0)$. 

In theory of effective models for a two-phased media, the case of a two-phased medium consisting
of an elastic carcass and a weakly viscous liquid is considered. The {\em passage to the limit} with respect to
a small size of a cell of periodicity and a small parameter in a viscosity coefficient leads to a
Biot poroelasticity system introduced and studied in \cite{{MAB}, {GN}, {ESP}}. This system also contains integrodifferential
summands of the convolution type. If we suppose additionally that the rigidity of the elastic carcass
tends to infinity as $\epsilon$  tending to $0$, then at the limit, we get an equation for sound pressure in a liquid that
is similar to an equation for a mixture of two liquids. However, in the representation for elements of
a matrix $D(t)$, there will be no summand $c_{0}^{ij}$ . We can also prove that the following conditions hold:
\begin{align}
  \sum_{k=1}^{\infty} c_{k}<\infty, \hspace{0.2cm} \sum_{k=1}^{\infty} \gamma_k c_{k}=\infty.
\end{align}

Note that the absence of the summand $c_{0}^{ij}$ affects the properties of the spectrum of eigen-oscillations of an effective medium. Continuous oscillations are not possible in this case.

Equations with structure and properties similar to the Gurtin-Pipkin
equation appear in the kinetic theory of gases. In this theory the
equations of a solid medium are derived from laws of pairwise
interaction of molecules. A series of equations for momenta can be derived from the Boltzmann equation by the
Grady method. Momenta are averaging of the distribution function of molecules by coordinates and
velocities with respect to velocity variables with certain weights. In particular, they are ordinary components of the Navier-Stokes equations of velocity, pressure, and density (as functions of spatial variables and time). They can be represented as momenta in a series of momentum equations.

It is known that the spectral properties of the Gurtin-Pipkin equation depend on
properties of the convolution kernel (see \cite{VRShamaev}). Moreover, note that if $K(0)$
and $K'(0)$ are finite, then the equation is the wave equation, but if $K(0)<\infty$
and $K'(0)=\infty$, the solutions and eigen-oscillations have stronger dissipative properties. However, the wave properties are always
preserved. Thus, the considered equation reflects the properties of media where there are oscillations
and dissipations simultaneously and the  scope of these properties is defined by the convolution kernel $K(t)$ at $t= 0$.

In \cite{VR}, Vlasov and Rautian analysed the system  (\ref{sytem1.1})$-$(\ref{sytem1.2}) with $\xi =1$, $k=0$, $m=2$. They established well-defined solvability of initial boundary value problems for the equations (\ref{sytem1.1}) and (\ref{sytem1.2}) in weighted Sobolev space on the positive semi-axis and examined some spectral propiertes of operator-valued functions. Their approach to the proof of the theorem about the well-defined solvability of the initial boundary value problem for the Gurtin-Pipkin equation essentially differs from that used by Pandolfi in \cite{P}. 
Unlike the argument given by Pandolfi,  Vlasov and Rautian considered the existence of solutions in weighted Sobolev space $W_{2, \gamma}^{2}(\mathbb{R}_{+}, A^{2})$ on the semi-axis $\mathbb{R}_{+}$. 

In the proof of existence theorem, Vlasov and Rautian efficiently utilize the Hilbert structure of the spaces $W_{2, \gamma}^{2}(\mathbb{R}_{+},
A^{2})$, $L_{2, \gamma}(\mathbb{R}_{+}, H)$ and the Paley-Wiener theorem (for more details, see \cite{{VR}, {VRShamaev}}).

In contrast to the results obtained in \cite{VR} and \cite{KVW}, our results of the present paper are more general, that  is to say, we
study asymptotic behavior of the spectrum of the abstract
Gurtin-Pipkin integro-differential equation with the kernel, depending on the parameter $\xi$, where $\xi \in (0, 1)$. On the semi-axis $\mathbb{R}_{+}=(0, \infty)$ we consider the problem (\ref{sytem1.1}), (\ref{sytem1.2}) with $\xi \in (0, 1)$, 
$k=0$ and $m=2$ and we proceed to analyze the structure of spectrum of operator $L(\zeta)$, 
{\small \begin{align}
L(\zeta):=\zeta^2 I+A^2-A^{2\xi}\widehat{K}(\zeta)
\end{align}} where
{\small \begin{align}
\widehat{K}(\zeta)=\sum_{k=1}^{\infty}\frac{c_k}{\zeta+\gamma_k}
\end{align}} is the Laplace transform of $K(t)$. 

Here, we consider the dependence of the asymptotic behavior of the complex roots as from the properties of the kernel of the equation
 $\hat{K}$ and parameter $\xi$. For the case  $0 < \xi < 1$ with $a_n\to \infty$, we have that the complex roots tend to the imaginary
axis, whereas in \cite{VR}, for the case $\xi=1$, the complex roots
tend to a line parallel to the imaginary axis. 
 
Throughout the paper, the expression $a \lesssim b$ stands for the inequality $a\leq Cb$ with a positive constant $C$; the expression $a \thickapprox b$ means that $a \lesssim b \lesssim a$.
\section{Statement of the Main Results} \label{statement-main-results} 
\subsection{Correct Solvability} Let $H$ be a separable Hilbert space and $A$ a self-adjoint positive operator in $H$ with a compact inverse. We associate the domain $\dom(A^{\beta})$ of the operator $A^{\beta}$, $\beta >0$, with a Hilbert space $H_{\beta}$ by introducing on $\dom(A^{\beta})$ the  norm $\|\cdot\|=\|A^{\beta}\cdot\|$, equivalent to the graph norm of the operator $A^{\beta}$. Denote by $\{e_j\}_{j=1}^{\infty}$ an othonormal basis formed by eigenvectors of $A$ corresponding to its eigenvalues $a_j$ such that $Ae_j=a_je_j, j \in \mathbb{N}$. The eigenvalues  $a_j$ are enumerated in increasing order with their multiplicity, that is, they satisfy: $0< a_1 \leq a_2 \leq \cdots \leq a_n\cdots$; where $a_n \to \infty$ as $n \to +\infty$.

By  $W_{2, \gamma}^{n}(\mathbb{R}_{+}, A^n)$ we denote the Sobolev space that consists of vector-functions on the semi-axis $\mathbb{R}_{+}=(0, \infty)$ with values in $H$ and norm  
\begin{align*}
\|u\|_{W_{2, \gamma}^{n}(\mathbb{R}_{+}, A^n)} \equiv \left(\int_{0}^{\infty}\exp(-2\gamma t)\left(\|u^{(n)}(t)\|^{2}_{H}+\|A^n u(t)\|^{2}_{H}\right) dt\right)^{1/2}, \hspace{0.3cm}\gamma \geq 0.
\end{align*}
A complete description of the space $W_{2, \gamma}^{n}(\mathbb{R}_{+}, A^n)$ and its properties are given in the monograph \cite[Chap. I]{LM}. Now, on the semi-axis $\mathbb{R}_{+}=(0, \infty)$ consider the problem
\begin{align}
\frac{d^2u}{dt^2}+A^2u- \int_{0}^{t}K(t-s)A^{2\xi} u(s) ds=f(t), \hspace{0.3cm} t \in \mathbb{R}_{+}\label{problem-valor-inicial-1},\\
u(+0)=\varphi_0, \hspace{0.4cm} u^{(1)}(+0)=\varphi_1,\hspace{0.4cm}  0<\xi<1.\label{problem-valor-inicial-2}
\end{align}
It is assumed that the vector-valued function $Af(t)$ belongs to $L_{2, \gamma_0}( \mathbb{R}_{+}, H)$ for some $\gamma_0 \geq 0$, and the scalar function $K(t)$ admits the representation 
\begin{align} 
K(t)=\sum_{j=1}^{\infty}c_j\exp(-\gamma_j t),
\end{align} where $c_j>0$, $\gamma_{j+1}>\gamma_{j}>0$, $j \in \mathbb{N}$, $\gamma_j \to +\infty$ $(j \to +\infty)$ and it is assumed that 
\begin{align}\label{condition-4}
\sum_{j=1}^{\infty}\frac{c_j}{\gamma_j}<1.
\end{align} Note that if the condition (\ref{condition-4}) is satisfied, then $K \in L_1(\mathbb{R}_{+})$ and $\|K\|_{L_1} <1$. Now,  if, moreover,  we take into consideration the condition 
\begin{align}\label{condition-5}
\sum_{j=1}^{\infty} c_j<+\infty,
\end{align} then the kernel $K$ belongs to the space $W_{1}^{1}(\mathbb{R}_{+})$.
\begin{definition} A vector-valued function $u$ is called a {\em strong solution} of problem (\ref{problem-valor-inicial-1}) and (\ref{problem-valor-inicial-2}) if for some $\gamma \geq 0$, $u \in W_{2, \gamma}^{2}(\mathbb{R}_{+}, A^2)$ satisfies the equation (\ref{problem-valor-inicial-1}) almost everywhere on the semi-axis $\mathbb{R}_{+}$, as well as the initial condition (\ref{problem-valor-inicial-2}).
\end{definition} 

In the paper \cite[Theorem 1]{VR}  was shown the existence of strong solution $u$ and the well-defined solvability of the problems  (\ref{problem-valor-inicial-1}) and (\ref{problem-valor-inicial-2}) for $\xi=1$. Here we mention only the results obtained there, since the proofs are similar for the case $0<\xi<1$. 
\begin{theorem} \label{theorem-about-solvability} Suppose for some $\gamma_0 \geq 0$, $A f(t) \in L_{2, \gamma_0}(\mathbb{R}_{+}, H)$. Suppose also that condition (\ref{condition-4}) is satisfied. Then
\begin{enumerate} [\ 1)]
\item If condition (\ref{condition-5}) holds and $\varphi_0 \in H_2$, $\varphi_1 \in H_1$, then for any $\gamma > \gamma_0$ the problems  (\ref{problem-valor-inicial-1}) and (\ref{problem-valor-inicial-2}) have a unique solution in the space $W_{2, \gamma}^{2}(\mathbb{R}_{+}, A^2)$ and this solution satisfies the estimate
\begin{align}
\|u\|_{W_{2, \gamma}^{2}(\mathbb{R}_{+}, A^2)} \leq d\left(\|A f\|^{2}_{L_{2, \gamma}(\mathbb{R}_{+}, H)}+\|A^2\varphi_0\|_{H}+\|A\varphi_1\|_{H}\right).
\end{align}  with a constant $d$ that does not depend on the vector-valued function $f$ and the vectors $\varphi_0$, $\varphi_1$.
\item If condition (\ref{condition-5}) does not hold and $\varphi_0 \in H_3$, $\varphi_1 \in H_2$, then for any $\gamma > \gamma_0$ the problems  (\ref{problem-valor-inicial-1}) and (\ref{problem-valor-inicial-2}) have a unique solution in the space $W_{2, \gamma}^{2}(\mathbb{R}_{+}, A^2)$ and this solution satisfies the estimate
\begin{align}
\|u\|_{W_{2, \gamma}^{2}(\mathbb{R}_{+}, A^2)} \leq d\left(\|A f\|^{2}_{L_{2, \gamma}(\mathbb{R}_{+}, H)}+\|A^3\varphi_0\|_{H}+\|A^2 \varphi_1\|_{H}\right).
\end{align} with a constant $d$ that does not depend on the vector-valued function $f$ and the vectors $\varphi_0$, $\varphi_1$.
\end{enumerate}
\end{theorem} 

\subsection{Spectral Analysis}\label{spectral-analysis} 
We now turn to the concepts of spectral analysis, which represent the basis of interest of this work. Consider the operator-valued function 
\begin{align}
L(\zeta)= \zeta^2I+A^2-A^{2\xi}\widehat{K}(\zeta), \hspace{0.4cm} 0 < \xi < 1,
\end{align} where
\begin{align*}
\widehat{K}(\zeta)=\sum_{k=1}^{\infty}\frac{c_k}{\zeta+\gamma_k}
\end{align*} is the Laplace transform of $K(t)$, and $I$ is the identity operator in $H$. The eigenvalues of the operator $A$ satisfy the inequalities $0< a_1<  a_2 < \cdots < a_n\cdots$; where $a_n \to \infty$ as $n \to +\infty$.

Let us now turn to the question about the structure of the spectrum of the operator-valued function $L(\zeta)$ in the case when the following condition holds: 
\begin{align} \label{suprem-de-las-gammas}
\sup_{k}\{\gamma_k(\gamma_{k+1}-\gamma_k)\}=+\infty
\end{align}

\begin{remark} The condition (\ref{suprem-de-las-gammas}) was used by Ivanov and Sheronova in \cite{ISH} to study the zeroes of the meromorphic function $\ell_n(\zeta)/\zeta$ in the case $a_n=n$ $ (n \in \mathbb{N})$. The same condition is used in the following theorem. 
 \end{remark}
\begin{theorem}\label{teorema-de-ceros-numerables} Suppose that the conditions (\ref{condition-5}) and (\ref{suprem-de-las-gammas}) are satisfied. Then the zeros of the meromorphic function  
\begin{equation}
    \ell_{n}(\zeta):= {\zeta}^2+ a^{2}_{n} \left(1- \frac{1}{a^{2(1-\xi)}_{n}}\sum_{k=1}^{\infty} \frac{c_k}{\zeta+\gamma_k}\right), \hspace{0.2cm}  0<\xi <1,
  \end{equation}
form a countable set of real zeroes  $\{\mu_{n, k}\}^{\infty}_{k=1}$, that satisfy the inequalities
\begin{align}\label{inequality-sum-finit}
   -\gamma_k<\mu_{n, k} <x_{n,k} < -\gamma_{k-1}, 
  \end{align}
\begin{align}
 \mu_{n, k}&=x_{n,k} + O\left( \frac{1}{a^{2(1-\xi)}_{n}}\right), \hspace{0.2cm}k \in \mathbb{N},\hspace{0.2cm} a_{n} \to \infty,\\
&\lim_{n\to\infty} \mu_{n, k}=\lim_{n\to\infty} x_{n,k}=-\gamma_k, 
  \end{align} where $x_{n,k}$ are real zeroes of the function
\begin{equation}
    f(\zeta):= 1- \frac{1}{a^{2(1-\xi)}_{n}}\sum_{k=1}^{\infty} \frac{c_k}{\zeta+\gamma_k}, \hspace{0.2cm}  0<\xi <1,
  \end{equation} 
together with a pair of complex-conjugate zeroes $\mu^{\pm}_{n}$,  $\mu^{+}_{n}= \overline{\mu^{-}_{n}}$ which admit the asymptotic representations when $a_{n} \to \infty$
{\small
\begin{align} 
&\mu^{\pm}_{n}(\xi)=-\frac{1}{2} \frac{1}{a_{n}^{2(1-\xi)}}\sum_{j=1}^{\infty}c_j+O\left(\frac{1}{a^{2(1-\xi)}_{n}}\right)\pm i\left(a_{n}+O\left(\frac{1}{a^{1-2\xi}_{n}}\right)\right),  \hspace{0.2cm}&\mbox{$0<\xi<\frac{1}{2}$}, \label{asymptota-sum-finita-para-mu1} \\
&\mu^{\pm}_{n}(\xi)=-\frac{1}{2} \frac{1}{a_{n}^{2(1-\xi)}}\sum_{j=1}^{\infty}c_j+O\left(\frac{1}{a^{2(1-\xi)}_{n}}\right)\pm i\left(a_{n}+O\left(a^{1-2\xi}_{n}\right)\right),  \hspace{0.2cm}&\mbox{$\frac{1}{2}<\xi<1$},\label{asymptota-sum-finita-para-mu2}\\
&\mu^{\pm}_{n}(\xi)=-\frac{1}{2} \frac{1}{a_{n}}\sum_{j=1}^{\infty}c_j+O\left(\frac{1}{a_{n}}\right)\pm i\left(a_{n}+O\left(1\right)\right),  \hspace{0.2cm}&\mbox{$\xi=\frac{1}{2}$}, \label{asymptota-sum-finita-para-mu3}
\end{align}} 
\end{theorem} 
The following theorem provides analysis about the asymptotic behavior of complex zeros of the function $\ell_{n}$, when the condition  $\sum_{k=1}^{\infty}c_k<+\infty$ is not satisfied, and the sequences $\{c_{k}\}^{\infty}_{k=1}$, $\{\gamma_{k}\}^{\infty}_{k=1}$ have the following asymptotic representation 
{\small
\begin{align}\label{asimptotas-de-c-y-gamma} 
   c_k=\frac{A}{k^{\alpha}} + O\left( \frac{1}{k^{\alpha +1}}\right), \hspace{0.2cm}  \gamma_k=B k^{\beta} + O( k^{\beta-1}), \hspace{0.2cm}  k \in \mathbb{N}
\end{align}} as $k \to \infty$, where  $c_k>0$, $\gamma_{k+1}>\gamma_{k}>0$, constants $A>0, B>0$,  $0<\alpha \leq 1$, $\alpha +\beta>1$ such that
 \begin{align}
   \sum_{k=1}^{\infty}\frac{c_k}{\gamma_k}<1.
\end{align} Note that from (\ref{asimptotas-de-c-y-gamma}),  
 \begin{align}\label{gama-k+1-menos-gama-k-1} 
 \gamma_{k+1}-\gamma_{k}=O(k^{\beta-1}),  \hspace{0.2cm}  k \to \infty.
\end{align}Thus for some $p \in \mathbb{N} $ such that $p \geq 1$ we have 
\begin{align}\label{gama-k+1-menos-gama-k-2}
\gamma_{k+1}-\gamma_{k} \thickapprox k^{\beta-p}.
\end{align} Moreover, if (\ref{gama-k+1-menos-gama-k-2}) is satisfied  as $k \to \infty$  with $p \geq 1$, then the condition  (\ref{suprem-de-las-gammas}) of Ivanov is satisfied when $p <2\beta$. 
\begin{remark} In the case $c_k=\frac{A}{k^{\alpha}}$, $\gamma_k=Bk^{\alpha}$, $a_n=n$, the asymptotic formulas for the complex zeroes $\lambda_{n}^{\pm}$ were provided by S. A. Ivanov in \cite{SAI}. 
\end{remark}

\begin{theorem} \label{teorema-de-ceros-infinitos} Assume that the conditions (\ref{asimptotas-de-c-y-gamma}) $-$ (\ref{gama-k+1-menos-gama-k-2}) are satisfied. Then the zeroes of the function $\ell_{n}$ form a countable set of real zeroes $\{\lambda_{n, k}|k \in \mathbb{N}\}$  such that
{\small \begin{equation} \label{desigualdad-ceros-infinitos}
  \cdots -\gamma_k<\lambda_{n, k} <x_{n, k} < -\gamma_{k-1}<\lambda_{n, k-1} <x_{n, k-1}<\cdots < -\gamma_{1}<\lambda_{n, 1} <x_{n, 1}<0 
  \end{equation}} together with a pair of complex-conjugate zeroes $\lambda^{\pm}_{n}$,  $\lambda^{+}_{n}= \overline{\lambda^{-}_{n}}$ which admit the asymptotic representations 
 \begin{align}
  &\lambda^{\pm}_{n}(\xi)=\pm ia_n-\frac{DA}{\beta B^{1-r}}\frac{1}{a^{1+r-2\xi}_{n}}+O\left( \frac{1}{a^{2(r-\xi)+1}_{n}}\right),\hspace{0.2cm} &\mbox{$0<r<\frac{1}{2}$}, \label{aimptotas-de-ceros-infinitos-1} \\
  &\lambda^{\pm}_{n}(\xi)=\pm ia_n-\frac{DA}{\beta B^{1-r}}\frac{1}{a^{1+r-2\xi}_{n}}+O\left( \frac{1}{a^{2(1-\xi)}_{n}}\right), \hspace{0.2cm}&\mbox{$\frac{1}{2}\leq r<1$}, \label{aimptotas-de-ceros-infinitos-2} \\
  &\lambda^{\pm}_{n}(\xi)=\pm ia_n-\frac{1}{2}\frac{A}{\beta}\frac{1}{a^{2(1-\xi)}_{n}}\ln a_n+ O\left( \frac{1}{a^{2(1-\xi)}_{n}}\right), \hspace{0.2cm}&\mbox{$r=1$}, \label{aimptotas-de-ceros-infinitos-3} 
  \end{align} as $a_n \to \infty$, where $r:=\frac{\alpha+\beta-1}{\beta}$, the constant $D$ depends of $r$ and  is defined as follows: 
 {\small
 \begin{align*} 
D:=\frac{i}{2}\int_{0}^{\infty}\frac{dt}{t^r(i+t)}=\frac{1}{2}\left(\int_{0}^{\infty}\frac{dt}{t^{r}(1+t^{2})}+i\int_{0}^{\infty}\frac{dt}{t^{r-1}(1+t^{2})}\right)=\frac{\pi}{2}\frac{\exp{\left(i\frac{\pi}{2}(1-r)\right)}}{\sin(\pi r)}. 
  \end{align*}}
  \end{theorem}
  
\section{Proof of Theorems \ref{teorema-de-ceros-numerables} and \ref{teorema-de-ceros-infinitos}}\label{proof-of-the-theorems} 
Before proving Theorem \ref{teorema-de-ceros-numerables}, first we shall prove the following statement. 

\begin{lemma}\label{lemma-de-ceros-finitos} Consider the function 
\begin{equation}
    \ell_{n,N}(\zeta):= {\zeta}^2+ a^{2}_{n} \left(1- \frac{1}{a^{2(1-\xi)}_{n}}\sum_{k=1}^{N} \frac{c_k}{\zeta+\gamma_k}\right), \hspace{0.2cm}  0<\xi <1,
  \end{equation}
The zeroes of the function $\ell_{n,N}$ form a set of real zeroes  $\{\mu_{n, k}\}^{N}_{k=1}$, such that
{\small \begin{equation}\label{desigualdad-ceros} 
   -\gamma_k<\mu_{n, k} <x_{n, k} < -\gamma_{k-1}<\mu_{n, k-1} <x_{n, k-1}<\cdots < -\gamma_{1}<\mu_{n, 1} <x_{n, 1}<0, 
  \end{equation}}
\begin{equation}
 \mu_{n, k}=x_{n, k} + O\left( \frac{1}{a^{2(1-\xi)}_{n}}\right), \hspace{0.2cm}k=1,2,\cdots N,\hspace{0.2cm} a_{n} \to \infty,
  \end{equation} where $x_{n, k}$ are real zeroes of the function
\begin{equation}
    f(\zeta):= 1- \frac{1}{a^{2(1-\xi)}_{n}}\sum_{k=1}^{N} \frac{c_k}{\zeta+\gamma_k}, \hspace{0.2cm}  0<\xi <1,
  \end{equation} together with a pair of complex-conjugate zeroes $\mu^{\pm}_{n}$,  $\mu^{+}_{n}= \overline{\mu^{-}_{n}}$, which admit the asymptotic representations when $a_{n} \to \infty$
{\small
\begin{align} 
&\mu^{\pm}_{n}(\xi)=-\frac{1}{2} \frac{1}{a_{n}^{2(1-\xi)}}\sum_{j=1}^{N}c_j+O\left(\frac{1}{a^{2(1-\xi)}_{n}}\right)\pm i\left(a_{n}+O\left(\frac{1}{a^{1-2\xi}_{n}}\right)\right),  \hspace{0.2cm}&\mbox{$0<\xi<\frac{1}{2}$}\label{asimptota-finita1}\\
&\mu^{\pm}_{n}(\xi)=-\frac{1}{2} \frac{1}{a_{n}^{2(1-\xi)}}\sum_{j=1}^{N}c_j+O\left(\frac{1}{a^{2(1-\xi)}_{n}}\right)\pm i\left(a_{n}+O\left(a^{1-2\xi}_{n}\right)\right),  \hspace{0.2cm}&\mbox{$\frac{1}{2}<\xi<1$}\label{asimptota-finita2}\\
&\mu^{\pm}_{n}(\xi)=-\frac{1}{2} \frac{1}{a_{n}}\sum_{j=1}^{N}c_j+O\left(\frac{1}{a_{n}}\right)\pm i\left(a_{n}+O\left(1\right)\right),  \hspace{0.2cm}&\mbox{$\xi=\frac{1}{2}$}\label{asimptota-finita3}
\end{align}} 
\end{lemma}

\begin{proof} According to The Fundamental Theorem of Algebra, the function $\ell_{n,N}$ has exactly $N+2$ zeroes, of which  $N$ are real zeroes (see graphically) and satisfy the inequality  (\ref{desigualdad-ceros}). 

Using the Vieta's formulas, we conclude that the  zeroes $x_{n, k}$ of the function $f(\zeta)$ satisfy the relations: 
{\small
\begin{equation} \label{ceros-usando-vieta1}
    \sum_{k=1}^{N} x_{n, k} =\left(\frac{1}{a^{2(1-\xi)}_{n}}\sum_{j=1}^{N} c_j \right)- \sum_{j=1}^{N} \gamma_j , \hspace{0.5cm} \prod_{k=1}^{N}x_{n, k}=(-1)^N \prod_{j=1}^{N}\gamma_j\left(1-\frac{S_N}{a^{2(1-\xi)}_{n}}\right)
\end{equation}} and zeroes $\mu_{n,k}$ of the function $\ell_{n,N}$ satisfy: 
{\small
\begin{equation} \label{ceros-usando-vieta2}
    \sum_{k=1}^{N+2}\mu_{n,k} = - \sum_{j=1}^{N} \gamma_j , \hspace{0.5cm} \prod_{k=1}^{N+2}\mu_{n,k}=(-1)^{N+2} a^{2}_{n}\prod_{j=1}^{N}\gamma_j\left(1-\frac{S_N}{a^{2(1-\xi)}_{n}}\right)
\end{equation}}
where
\begin{equation} 
    S_N:=\sum_{j=1}^{N} \frac{c_j}{\gamma_j}. 
\end{equation} From (\ref{ceros-usando-vieta1}) and (\ref{ceros-usando-vieta2}), we obtain
{\small
\begin{equation*}
\sum_{k=1}^{N+2}\mu_{n,k}= \sum_{k=1}^{N}\left(x_{n, k}+O\left(\frac{1}{a^{2(1-\xi)}_{n}}\right)\right)+2 \re \mu^{\pm}_{n}
\end{equation*}}
{\small
\begin{equation*} 
 \prod_{k=1}^{N+2}\mu_{n,k}=\prod_{k=1}^{N}\left(x_{n, k}+O\left(\frac{1}{a^{2(1-\xi)}_{n}}\right)\right) \mu^{+}_{n}\mu^{-}_{n}, 
\end{equation*}} Here we set $\mu_{n, N+1}=\mu^{+}_{n}$ and $\mu_{n,N+2}=\mu^{-}_{n}$. Consequently, 
{\small
\begin{equation*}
 \re \mu^{\pm}_{n}= -\frac{1}{2} \frac{1}{a^{2(1-\xi)}_{n}}\sum_{j=1}^{N}c_j+O\left(\frac{1}{a^{2(1-\xi)}_{n}}\right), \hspace{0.2cm} (\re \mu^{\pm}_{n})^2+(\im \mu^{\pm}_{n})^2=a^{2}_{n}+O\left(\frac{1}{a^{-2\xi}_{n}}\right).
\end{equation*}} 
 Observe that, as $a_n \to \infty$ the following holds  
{\small
\begin{align*} 
\sqrt{a^{2}_{n}+O\left(\frac{1}{a^{-2\xi}_{n}}\right)}&= \pm a_n\left(1+\frac{1}{a^{2}_{n}}O\left(\frac{1}{a^{-2\xi}_{n}}\right)\right)^{1/2}=\pm a_n\left(1+\frac{1}{2a^{2}_{n}}O\left(\frac{1}{a^{-2\xi}_{n}}\right)\right)\\
&=\pm a_n+O\left(\frac{1}{a^{1-2\xi}_n}\right), 
\end{align*}} Therefore
\begin{align*} 
\im \mu^{\pm}_{n}=\pm a_{n}+O\left(\frac{1}{a^{1-2\xi}_{n}}\right).
\end{align*} And so when $a_{n} \to \infty$ we have asymptotic formulas (\ref{asimptota-finita1}), (\ref{asimptota-finita2}) and (\ref{asimptota-finita3}). 

Proof of Theorem \ref{teorema-de-ceros-numerables}. Consider the case $\sum_{k=1}^{\infty}c_k < + \infty$. From this convergence follows that for any   $\epsilon >0$ we can find $N$ such that
\begin{align} 
\sum_{k=N+1}^{\infty}c_k<\frac{\epsilon}{10}.
\end{align}
Consider a circle with centre  $\mu^{+}_{n}:D_{\epsilon}(\mu^{+}_{n})=\{\zeta:\zeta=\mu^{+}_{n} + \epsilon e^{i \varphi}, 0 \leq \varphi <2\pi \}$ and radius $\epsilon$ (for zero $\mu^{-}_{n}$ the argument is completely analogous).

Set the function $\ell_n(\zeta)$ in the form
\begin{align} 
\ell_n(\zeta)=\ell_{n, N}(\zeta)+ m_{n, N}(\zeta),
\end{align}  
where
\begin{align*} 
\ell_{n,N}(\zeta) &= {\zeta}^2 + a^{2}_{n} \left(1- \frac{1}{a^{2(1-\xi)}_{n}}\sum_{k=1}^{N} \frac{c_k}{\zeta+\gamma_k}\right), \hspace{0.2cm} &\mbox{$0<\xi <1$},\\
m_{n,N}(\zeta) &= -\frac{1}{a^{-2\xi}_{n}}\sum_{k=N+1}^{\infty} \frac{c_k}{\zeta+\gamma_k}, \hspace{0.2cm}  &\mbox{$0<\xi <1$}.
\end{align*}Choose a small $\epsilon > 0$ such that in the neighborhood $B_{\epsilon}({\mu}^{+}_{n})= \{\zeta:|\zeta-{\mu}^{+}_{n}| <\epsilon \}$ there are no other zeroes of the function $\ell_{n, N}(\zeta)$ and so we shall estimate the function $\ell_{n, N}(\zeta)$  on the circle $D_{\epsilon}({\mu}^{+}_{n})$. For this purpose, note that for  $\zeta \in D_{\epsilon}({\mu}^{+}_{n})$, we obtain
{\small 
\begin{align*} 
|\zeta-\mu_{n, k}| \geq \left(\left(-\frac{1}{2}\frac{1}{a^{2(1-\xi)}_{n}}\sum_{k=1}^{N} c_k - \mu_{n, k} + \epsilon \cos (\varphi)\right)^2+ (a_{n}+\epsilon \sin (\varphi))^2\right)^{1/2}\geq a_n -\epsilon. 
\end{align*}}Moreover, observe that for $\zeta \in D_{\epsilon}({\mu}^{+}_{n})$,
\begin{align*} 
|\zeta-\mu^{+}_{n}| = \epsilon;   \hspace{1cm} |\zeta-\mu^{-}_{n}| > 2a_n-\epsilon.
\end{align*} 
From the above inequalities, we obtain the following lower bound: 
\begin{align} \label{prodcuto-de-ceros1}
\left|\prod_{k=1}^{N}(\zeta-\mu_{n, k})(\zeta-\mu^{+}_{n})(\zeta-\mu^{-}_{n})\right|_{\zeta \in D_{\epsilon}({\mu}^{+}_{n})} \geq (a_n-\epsilon)^{N+1}\epsilon.
\end{align} From inequality  
\begin{align*} 
|\zeta + \gamma_j|_{\zeta \in D_{\epsilon}(\mu^{+}_{n})} & \leq \left|\left(-\frac{1}{2}\frac{1}{a^{2(1-\xi)}_{n}}\sum_{k=1}^{N} c_k\right) + \gamma_j + \epsilon \cos(\varphi) \right|+ |a_{n}+\epsilon \sin (\varphi)|\\
 &\leq \left(\frac{1}{2}\frac{1}{a^{2(1-\xi)}_{n}}\sum_{k=1}^{N} c_k \right)+ \gamma_j + \epsilon + a_{n}+\epsilon\\
&\leq a_n+2\epsilon+M_1, 
\end{align*}
 where $M_1 = \left(\frac{1}{2}\frac{1}{a^{2(1-\xi)}_{n}}\sum_{k=1}^{\infty} c_k\right) + \gamma_j$, we come to the estimate 
\begin{align}\label{producto-de-ceros2}
\left|\prod_{k=1}^{N}(\zeta- \gamma_k)\right|_{\zeta \in D_{\epsilon}({\mu}^{+}_{n})}\leq (a_n+2\epsilon+M_1)^N.
\end{align} Now, from (\ref{prodcuto-de-ceros1}) and (\ref{producto-de-ceros2}) we get
\begin{align*}
|\ell_{n,N}(\zeta)|&= \frac{|\prod_{k=1}^{N}(\zeta-\mu_{n, k})(\zeta-\mu^{+}_{n})(\zeta-\mu^{-}_{n})|}{|\prod_{k=1}^{N}(\zeta- \gamma_k)|}\\
&\geq \frac{(a_n-\epsilon)^{N+1}\epsilon}{(a_n+2\epsilon+M_1)^N}\\
&=\epsilon a_n\left(\left(1-\frac{N(3\epsilon+M_1)+\epsilon}{a_n}\right)+O\left(\frac{1}{a^{2}_n}\right)\right).
\end{align*} Hence, for a $a_n$ sufficiently large, the following inequality holds 
\begin{align}\label{cota-por-abajo}
|\ell_{n,N}(\zeta)| \geq \frac{\epsilon a_n}{2}. 
\end{align} Now, for a $a_n$ sufficiently large, we estimate the upper bound of the function $m_{n,N}(\zeta)$: 
{\small
\begin{align}\label{cota-por-arriba}
|m_{n,N}(\zeta)| &\leq  a^{2 \xi}_{n}\sum_{k=N+1}^{\infty} \frac{c_k}{|\zeta+\gamma_k|}\nonumber \\ 
&\leq a^{2 \xi}_{n}\sum_{k=N+1}^{\infty} \frac{c_k}{a_n-\epsilon},  \hspace{0.4cm}\zeta \in D_{\epsilon}({\mu}^{+}_{n})\\
&\leq \frac{a^{2 \xi}_{n}}{(a_n-\epsilon)}\frac{\epsilon}{10}\leq a^{2 \xi -1}_{n}\frac{\epsilon}{5} < a_n\frac{\epsilon}{5}, \hspace{0.4cm} 0<\xi < 1.\nonumber
\end{align}} 
That is, from inequality (\ref{cota-por-abajo}) and (\ref{cota-por-arriba}) we get
\begin{align}
|\ell_{n,N}(\zeta)| \geq \frac{\epsilon a_n}{2} > a_n\frac{\epsilon}{5} > |m_{n,N}(\zeta)|
\end{align} and so, using the Rouche's theorem we conclude that in the circle $D_{\epsilon}({\mu}^{+}_{n})$  for a $a_n$ sufficiently large, the function $\ell_n(\zeta)$ has one simple zero $\mu^{+}_{n}$. 

Continuing with the proof of  Theorem \ref{teorema-de-ceros-numerables}, let us first prove that the set of zeroes of the function 
\begin{align} \label{funtion-meromorfhic-zerousnumerables}
 \frac{\ell_{n}(\zeta)}{a^{2}_{n}}= \frac{{\zeta}^2}{a^{2}_{n}}+ \left(1- \frac{1}{a^{2(1-\xi)}_{n}}\sum_{k=1}^{\infty} \frac{c_k}{\zeta+\gamma_k}\right), \hspace{0.2cm}  0<\xi <1,
\end{align} coincides with the union of countable series of zeroes $\{\mu_{n, k}|k \in \mathbb{N}\}$ and a pair of complex-conjugate zeroes $\mu^{\pm}_{n}$,  $\mu^{+}_{n}= \overline{\mu^{-}_{n}}$. Moreover, let us prove that  real zeroes $\{\mu_{n, k}|k \in \mathbb{N}\}$ satisfy the inequalities (\ref{inequality-sum-finit}). Secondly, we prove that for the complex-conjugate zeroes $\mu^{\pm}_{n}$, asymptotic formulas (\ref{asymptota-sum-finita-para-mu1}), (\ref{asymptota-sum-finita-para-mu2}) and (\ref{asymptota-sum-finita-para-mu3}) indeed hold. 

For each $n \in \mathbb{N}$ fixed, consider on the complex plane a rectangular contour 
$\Gamma = \{\Gamma^{+}\cup\Gamma^{-}\cup \Gamma^{*} \cup 
\bar{\Gamma}^{*}\}$, where
\begin{align*} 
  &\Gamma^{\pm}=\{\zeta \in \mathbb{C}:\re \zeta=\pm X, |\im \zeta|\leq Y, X>0,  Y>0\},\\
  &\Gamma^{*}=\{\zeta \in \mathbb{C}:|\re \zeta|\leq X, \im \zeta=\pm Y, X>0, Y>0 \},\\
  &\bar{\Gamma}^{*} =\{\zeta \in \mathbb{C}:\bar{\zeta} \in \Gamma^{*}\}.
  \end{align*} Set 
\begin{align}\label{functions-meromorphic-f-and-g}
f(\zeta)=1- \frac{1}{a^{2(1-\xi)}_{n}}\hat{K}(\zeta), \hspace{1cm}  g(\zeta)= \frac{{\zeta}^2}{a^{2}_{n}}, \hspace{1cm}  0<\xi <1,
\end{align}where $\hat{K}(\zeta)= \sum_{k=1}^{\infty} \frac{c_k}{\zeta+\gamma_k}$. Similarly to the proof given in \cite{VR}, we can choose $X>0, Y>0$ such that for all $\zeta \in \Gamma$, 
\begin{align}\label{desigualdad-f-y-g}
|f(\zeta)|< |g(\zeta)|.   
\end{align} We will prove that for every side of rectangular contour, the inequality (\ref{desigualdad-f-y-g}) holds. For all $\zeta \in \Gamma^{-}$, we have 
the estimate 
{\small
\begin{align*}
|f(\zeta)|& =|1- a^{2(\xi-1)}_{n}\hat{K}(\zeta)|= \left|1- \frac{1}{a^{2(1-\xi)}_{n}}\hat{K}(\zeta)\right| \leq 1+ \left|\frac{1}{a^{2(1-\xi)}_{n}}\hat{K}(\zeta)\right|\\
&=1+ \frac{1}{a^{2(1-\xi)}_{n}}\left|\sum_{k=1}^{\infty} \frac{c_k}{-X+iy+\gamma_k} \right|\leq 1+ \frac{1}{a^{2(1-\xi)}_{n}}\sum_{k=1}^{\infty} \frac{c_k}{|-X+iy+\gamma_k|}\\
&\leq 1+ \frac{1}{a^{2(1-\xi)}_{n}}\sum_{k=1}^{\infty} \frac{c_k}{|-X+\gamma_k|}:=q(X),  \hspace{1cm}  0<\xi <1.
\end{align*}
} Set $X=X_N:=(\gamma_N+\gamma_{N+1})/2$, where $N$ is a sufficiently large natural number, we show that
\begin{align}
\inf_{N}\left\{ \frac{q(X_N)}{X^{2}_{N}}\right \}=0.  
\end{align} Note that $q(X_N)$ is decomposed as 
{\small
\begin{align*}
 q(X_N)=\frac{1}{a^{2(1-\xi)}_{n}}\left(\sum_{k=1}^{N} \frac{c_k}{X_N-\gamma_k}+ \sum_{k=N+1}^{\infty} \frac{c_k}{\gamma_k-X_N}\right)+1.
\end{align*}}
 Let $\delta_N:=\gamma_{N+1}-\gamma_{N}$. Since
{\small
\begin{align*}
X_N-\gamma_k&= \frac{1}{2}(\gamma_{N+1}-\gamma_k+\gamma_{N}-\gamma_k)\\
&\geq \frac{1}{2}(\gamma_{N+1}-\gamma_k)\geq \frac{1}{2}\delta_N, \hspace{0.3cm} k=1,\dots, N\\              
 \gamma_k-X_N&\geq  \frac{1}{2}(\gamma_{k}-\gamma_N)\geq  \frac{1}{2}\delta_N, \hspace{0.3cm} k=N+1, N+2, \dots            
\end{align*}} we see that
\begin{align*}
q(X_N)&\leq \frac{2}{a^{2(1-\xi)}_{n}}\left(\sum_{k=1}^{N} \frac{c_k}{\gamma_{N+1}-\gamma_k}+ \sum_{k=N+1}^{\infty} \frac{c_k}{\gamma_k-\gamma_N}\right)+1\\
&\leq \left(\frac{2}{\delta_N a^{2(1-\xi)}_{n}}\sum_{k=1}^{\infty} c_k\right)+1:=\left(\frac{2S_1}{\delta_Na^{2(1-\xi)}_{n}}\right)+1
\end{align*} and
 \begin{align}
\frac{q(X_N)}{X^{2}_N}\leq \frac{1}{\gamma_{N}^{2}}\left(\frac{2S_1}{\delta_Na^{2(1-\xi)}_{n}}+1\right).
\end{align} Taking into account the condition  (\ref{suprem-de-las-gammas}) and for $N$ sufficiently large, we get 
 \begin{align*}
\inf_{N}\left\{ \frac{q(X_N)}{X^{2}_{N}}\right \}=0.  
\end{align*} It follows that for a given $a_n$, the inequality 
\begin{align}
\frac{1}{a^{2}_n}> \frac{1}{\gamma_{N}^{2}}\left(\frac{2S_1}{\delta_Na^{2(1-\xi)}_{n}}+1\right)
\end{align} holds for $N$ sufficiently large. Then, for all $\zeta \in \Gamma^{-}$, we get
{\small
\begin{align}
|g(\zeta)|\geq  \frac{X^{2}_{N}}{a^{2}_{N}}> X^{2}_{N}\left( \frac{1}{\gamma_{N}^{2}}\left(\frac{2S_1}{\delta_Na^{2(1-\xi)}_{n}}+1\right)\right)\geq X^{2}_{N}\frac{q(X_N)}{X^{2}_N}\geq|f(\zeta)|. 
\end{align}} For all $\zeta \in \Gamma^{+}$, we have
{\small
\begin{align*}
|f(\zeta)|& \leq 1+ \frac{1}{a^{2(1-\xi)}_{n}}\sum_{k=1}^{\infty} \frac{c_k}{|X_N+iy+\gamma_k|}\\
 &\leq 1+ \frac{1}{X_Na^{2(1-\xi)}_{n}}\sum_{k=1}^{\infty} c_k :=1+\frac{S_1}{X_Na^{2(1-\xi)}_{n}}<+\infty, \hspace{2cm}  0<\xi <1.
\end{align*}} Hence, if we choose  $X=X_N$  such that 
\begin{align*}
\frac{X_N^2}{a^{2}_{n}}>1+\frac{S_1}{X_Na^{2(1-\xi)}_{n}}
\end{align*}then 
\begin{align*}
X_N^3- a^{2}_{n}X_N-\frac{a^{2}_{n}}{a^{2(1-\xi)}_{n}}S_1 >0. 
\end{align*}Let $X_N>1$, then 
\begin{align*}
X_N^3- a^{2}_{n}X_N-a^{2}_{n}S_1> X_N^3- a^{2}_{n}X_N-\frac{a^{2}_{n}}{a^{2(1-\xi)}_{n}}S_1X_N>0 
\end{align*}holds, if we choose $X_N>a_n\sqrt{1+\frac{S_1}{a^{2(1-\xi)}_{n}}}$. Hence, $|f(\zeta)|<|g(\zeta)|$  for all $\zeta  \in \Gamma^{+}$.

Now, consider a horizontal segment of the contour  
\begin{align*} 
  \Gamma^{*}=\{\zeta \in \mathbb{C}:|\re \zeta|\leq X_N, \im \zeta=+ Y\}.
   \end{align*} It follows that, for all  $\zeta  \in \Gamma^{*}$ we get
\begin{align}
\inf_{\zeta  \in \Gamma^{*}}\left| \frac{\zeta^{2}}{a^{2}_{n}}\right|=\frac{Y^2}{a^{2}_{n}}.
\end{align} 
For all $\zeta  \in \Gamma^{*}$ we have
{\small
\begin{align*}
|f(\zeta)|& \leq 1+ \frac{1}{a^{2(1-\xi)}_{n}}\sum_{k=1}^{\infty} \frac{c_k}{|x+iY+\gamma_k|}\\
 &\leq 1+ \frac{1}{Ya^{2(1-\xi)}_{n}}\sum_{k=1}^{\infty} c_k:=1+\frac{S_1}{Ya^{2(1-\xi)}_{n}} <+\infty, \hspace{2cm}  0<\xi <1.
\end{align*}} Hence, if we choose  $Y$  such that 
\begin{align*}
\frac{Y^2}{a^{2}_{n}}>1+\frac{S_1}{Ya^{2(1-\xi)}_{n}} 
\end{align*}then 
\begin{align*}
Y^3- a^{2}_{n}Y-\frac{a^{2}_{n}}{a^{2(1-\xi)}_{n}}S_1 >0. 
\end{align*}Let $Y>1$, then 
\begin{align*}
Y^3- a^{2}_{n}Y-a^{2}_{n}S_1> Y^3- a^{2}_{n}Y-\frac{a^{2}_{n}}{a^{2(1-\xi)}_{n}}S_1Y>0 
\end{align*}holds, if we choose $Y>a_n\sqrt{1+\frac{S_1}{a^{2(1-\xi)}_{n}}}$. Hence, the inequality (\ref{desigualdad-f-y-g})  holds for all $\zeta  \in \Gamma^{*}$. 

 In  case of  $\zeta  \in \bar{\Gamma}^{*}$, the inequality (\ref{desigualdad-f-y-g}) holds under the same condition since $\overline{f(\zeta)}=f(\bar{\zeta})$ and  $\overline{g(\zeta})=g(\bar{\zeta})$. Thus, since $|f(\zeta)|<|g(\zeta)|$ for all $\zeta  \in \Gamma$, then
by the Rouche's theorem and argument principle, we obtain
 \begin{align*}
N(g)-P(g)=N(f+g)-P(f+g),
\end{align*} where $N$ and $P$ denote, respectively,  the number of zeroes and poles inside the contour $\Gamma$, with each zero and pole counted as many times as its order or multiplicity, respectively, indicate. By the definition of the function $g(\zeta)$, $N(g)-P(g)=2$. Inside the contour $\Gamma$, the function $f+g$ has $N$ poles: $ -\gamma_N, -\gamma_{N-1},...,-\gamma_{1}$. Therefore,  $N(f+g)= N+2$. It is not difficult to note graphically (see also \cite{ISH}) that  the function $f+g$ has $N+1$ or $N$ real zeroes inside the contour  $\Gamma$, which satisfy inequality (\ref{desigualdad-ceros-infinitos}), and depending on whether $\mu_{N,n}<X_N$ or $\mu_{N,n}\geq X_N$. But if $\mu_{N,n}<X_N$, then $f+g$ has one complex zero inside $\Gamma$ and this is not possible since $\overline{(f+g)}(\zeta)=(f+g)(\overline{\zeta})$. Consequently, there are exactly two complex zeroes  $\mu^{\pm}_{n}$ inside $\Gamma$, where  $\mu^{+}_{n}=\overline{\mu^{-}_{n}}$.

Proof of Theorem \ref{teorema-de-ceros-infinitos}. As in the Theorem \ref{teorema-de-ceros-numerables} we consider the same function (\ref{funtion-meromorfhic-zerousnumerables}) and we prove that the set of its zeroes coincides with the union of countable series of zeroes $\{\lambda_{n, k}|k \in \mathbb{N}\}$ and a pair of complex-conjugate zeroes $\lambda^{\pm}_{n}$,  $\lambda^{+}_{n}= \overline{\lambda^{-}_{n}}$. Moreover, we prove that  real zeroes $\{\lambda_{n, k}|k \in \mathbb{N}\}$ satisfy the inequalities (\ref{desigualdad-ceros-infinitos}). Secondly, we prove that for the complex-conjugate zeroes $\lambda^{\pm}_{n}$, asymptotic formulas (\ref{aimptotas-de-ceros-infinitos-1})$-$(\ref{aimptotas-de-ceros-infinitos-3}) indeed hold. 

For each fixed natural number, we will prove that for every side of rectangular contour $\Gamma = \{\Gamma^{+}\cup\Gamma^{-}\cup \Gamma^{*} \cup \bar{\Gamma}^{*}\}$, the inequality (\ref{desigualdad-f-y-g}) holds. The $f(\zeta)$ and $g(\zeta)$ are the same functions as in (\ref{functions-meromorphic-f-and-g}). Now, for all $\zeta \in \Gamma^{-}$, we have the same 
estimate 
{\small
\begin{align*}
|f(\zeta)| \leq 1+ \frac{1}{a^{2(1-\xi)}_{n}}\sum_{k=1}^{\infty} \frac{c_k}{|-X+\gamma_k|}:=q(X),  \hspace{1cm}  0<\xi <1.
\end{align*}
} Set $X=X_N:=(\gamma_N+\gamma_{N+1})/2$, where $N$ is a sufficiently large natural number, we have  
\begin{align}
\inf_{N}\left\{ \frac{q(X_N)}{X^{2}_{N}}\right \}=0.  
\end{align} Indeed, we note that
{\small
\begin{align*}
 q(X_N)=\frac{1}{a^{2(1-\xi)}_{n}}\left(\sum_{k=1}^{N} \frac{c_k}{X_N-\gamma_k}+ \sum_{k=N+1}^{\infty} \frac{c_k}{\gamma_k-X_N}\right)+1,
\end{align*}} then since
{\small
\begin{align*}
X_N-\gamma_k&\geq \frac{1}{2}(\gamma_{N+1}-\gamma_k), \hspace{0.3cm} k=1,\dots, N\\             
 \gamma_k-X_N&\geq  \frac{1}{2}(\gamma_{k}-\gamma_N), \hspace{0.3cm} k=N+1, N+2, \dots                 
\end{align*}} it follows that 
{\small
\begin{align*}
 q(X_N)&\leq \frac{2}{a^{2(1-\xi)}_{n}}\left(\sum_{k=1}^{N} \frac{c_k}{\gamma_{N+1}-\gamma_k}+ \sum_{k=N+1}^{\infty} \frac{c_k}{\gamma_k-\gamma_N}\right)+1\\
&=\frac{2}{a^{2(1-\xi)}_{n}}\left(\sum_{k=1}^{N} \frac{c_k}{\gamma_k}\left(\frac{\gamma_{N+1}}{\gamma_k}-1\right)^{-1}+ \sum_{k=N+1}^{\infty} \frac{c_k}{\gamma_k}\left(1-\frac{\gamma_N}{\gamma_k}\right)^{-1}\right)+1.
\end{align*}} By the above hypothesis (see (\ref{gama-k+1-menos-gama-k-1}) and  (\ref{gama-k+1-menos-gama-k-2})), where   for
{\small
\begin{align*}
\gamma_{k+1}-\gamma_{k}=B (k+1)^{\beta}- Bk^{\beta}+ O( k^{\beta-1}) \thickapprox k^{\beta-p}, \hspace{0.2cm} 1\leq p <\infty,  \hspace{0.2cm}k \to \infty,
\end{align*}} $p/2 <\beta$ holds, we conclude that for all $k=1, \dots, N$,
{\small 
\begin{align*}
\left(\frac{\gamma_{N+1}}{\gamma_k}-1\right)^{-1}&\leq \left(\frac{\gamma_{N+1}}{\gamma_N}-1\right)^{-1}= \frac{\gamma_{N}}{\gamma_{N+1}-\gamma_N}\\
&\leq C_1N^{p-\beta}\left(BN^{\beta}+O\left(N^{\beta -1}\right)\right)\\
&=C_1N^{p}\left(B+O\left(\frac{1}{N}\right)\right)=C_1N^{p}B.
\end{align*}} Similarly, we can prove that for all $k=N+1, N+2 \dots $ the following expression holds.
{\small 
\begin{align*}
\left(1-\frac{\gamma_{N}}{\gamma_k}\right)^{-1}&\leq \left(1-\frac{\gamma_{N}}{\gamma_{N+1}}\right)^{-1}=\frac{\gamma_{N+1}}{\gamma_{N+1}-\gamma_N}\\
&\leq  C_1N^{p-\beta}\left(B(N+1)^{\beta}+O\left((N+1)^{\beta -1}\right)\right)\\
&=C_1N^p\left(1+O\left(\frac{1}{N}\right)\right)\left(B+O\left(\frac{1}{N+1}\right)\right)\\ 
&=C_1BN^p, \hspace{0.5cm} N \to \infty.
\end{align*}} Here and in what follows,  $C_1>0$. Then
\begin{align}
q(X_N)\leq \frac{2C_1N^{p}B}{a^{2(1-\xi)}_{n}}\left(\sum_{k=1}^{\infty} \frac{c_k}{\gamma_k}\right)+1=\frac{2C_1N^{p}B}{a^{2(1-\xi)}_{n}}S+1
\end{align}and
 \begin{align}
\frac{q(X_N)}{X^{2}_N}&\leq \frac{\frac{2C_1N^{p}B}{a^{2(1-\xi)}_{n}}S+1}{\gamma^{2}_N}=\frac{\frac{2C_1N^{p}B}{a^{2(1-\xi)}_{n}}S+1}{(BN^{\beta}+O(N^{\beta-1}))^2}\\
&=\frac{\frac{2C_1N^{p-2\beta}B^{-1}}{a^{2(1-\xi)}_{n}}S+O\left(N^{-2\beta}\right)}{1+O\left(\frac{1}{N}\right)+O\left(\frac{1}{N^2}\right)}.
\end{align} This equation converges to zero as $N \to \infty$ since $2\beta>p$. This means that  
 \begin{align*}
\inf_{N}\left\{ \frac{q(X_N)}{X^{2}_{N}}\right \}=0.  
\end{align*} It follows that for a given $a_n$, the inequality 
\begin{align}
\frac{1}{a^{2}_n}>\frac{2C_1N^{p-2\beta}B^{-1}}{a^{2(1-\xi)}_{n}}S+O(N^{-2\beta})
\end{align} holds for $N$ sufficiently large. Then, for all $\zeta \in \Gamma^{-}$, we get
{\small
\begin{align}
|g(\zeta)|\geq  \frac{X^{2}_{N}}{a^{2}_{n}}> X^{2}_{N}\left(\frac{2C_1N^{p-2\beta}B^{-1}}{a^{2(1-\xi)}_{n}}S+O(N^{-2\beta})\right)\geq X^{2}_{N}\frac{q(X_N)}{X^{2}_N}\geq|f(\zeta)|. 
\end{align}} For all $\zeta \in \Gamma^{+}$,
{\small
\begin{align*}
|f(\zeta)|& \leq 1+ \frac{1}{a^{2(1-\xi)}_{n}}\sum_{k=1}^{\infty} \frac{c_k}{|X+iy+\gamma_k|}\\
 &\leq 1+ \frac{1}{a^{2(1-\xi)}_{n}}\sum_{k=1}^{\infty} \frac{c_k}{|X+\gamma_k|} \\
&  \leq 1+ \frac{1}{a^{2(1-\xi)}_{n}}\sum_{k=1}^{\infty} \frac{c_k}{\gamma_k}:=1+\frac{S}{a^{2(1-\xi)}_{n}}<+\infty,  \hspace{2cm}  0<\xi <1.
\end{align*}} And so, for all $\zeta \in \Gamma^{+}$, the following is satisfied
{\small
\begin{align*}
|g(\zeta)|\geq  \frac{X^{2}_{N}}{a^{2}_{n}}> 1+ \frac{1}{a^{2(1-\xi)}_{n}}\sum_{k=1}^{\infty} \frac{c_k}{\gamma_k} \geq |f(\zeta)|
\end{align*}} if $X_N>a_n\sqrt{\frac{S}{a^{2(1-\xi)}_{n}}+1}$. 

Now, consider a horizontal segment of the contour  
\begin{align*} 
  \Gamma^{*}=\{\zeta \in \mathbb{C}:|\re \zeta|\leq X_N, \im \zeta=+ Y\}.
   \end{align*} It follows that, for all  $\zeta  \in \Gamma^{*}$ we get
\begin{align}
\inf_{\zeta  \in \Gamma^{*}}\left| \frac{\zeta^{2}}{a^{2}_{n}}\right|=\frac{Y^2}{a^{2}_{n}}.
\end{align} Let us denote
\begin{align*} 
  \Gamma^{*}_{1}=\{\zeta \in \mathbb{C}:0 \leq \re \zeta \leq X_N, \im \zeta=+Y\},\\
  \Gamma^{*}_{2}=\{\zeta \in \mathbb{C}:-X_N \leq \re \zeta \leq 0, \im \zeta=+Y\}.
   \end{align*}
For all $\zeta  \in \Gamma^{*}_{1}$ we have
{\small
\begin{align*}
|f(\zeta)|& \leq 1+ \frac{1}{a^{2(1-\xi)}_{n}}\sum_{k=1}^{\infty} \frac{c_k}{|x+iY+\gamma_k|}\\
 &\leq 1+ \frac{1}{a^{2(1-\xi)}_{n}}\sum_{k=1}^{\infty} \frac{c_k}{\gamma_k}\left|\frac{x+iY}{\gamma_k}+1\right|^{-1}\\
&  \leq 1+ \frac{1}{a^{2(1-\xi)}_{n}}\sum_{k=1}^{\infty} \frac{c_k}{\gamma_k}:=1+\frac{S}{a^{2(1-\xi)}_{n}},  \hspace{3cm}  0<\xi <1.
\end{align*}} Hence, if we choose  $Y$  such that $Y>a_n\sqrt{\frac{S}{a^{2(1-\xi)}_{n}}+1}$,  then $\forall \zeta  \in \Gamma^{*}_{1}$ the inequality (\ref{desigualdad-f-y-g}) holds. Note that for all  $\zeta  \in \Gamma^{*}_{2}\cap[\zeta_1, \zeta_2]$, where $\zeta_1=-\gamma_{M+1}+iY, \zeta_2= -\gamma_{M}+iY$, $1\leq M < N$, $M \in \mathbb{N}$, immediately we have the following inequality:
{\small
\begin{align*}
|f(\zeta)|& \leq 1+ \frac{1}{a^{2(1-\xi)}_{n}}\sum_{k=1}^{\infty} \frac{c_k}{|x+iY+\gamma_k|}\\
 &\leq 1+ \frac{1}{a^{2(1-\xi)}_{n}}\sum_{k=1}^{\infty} \frac{c_k}{\gamma_k}\left|\frac{x+iY}{\gamma_k}+1\right|^{-1}\\
&  \leq 1+ \frac{1}{a^{2(1-\xi)}_{n}}\left(\sum_{k=1}^{M-1} +\sum_{k=M+2}^{\infty}\right) \frac{c_k}{\gamma_k}\left|\frac{x}{\gamma_k}+1\right|^{-1}+\frac{1}{a^{2(1-\xi)}_{n}}\frac{c_M+c_{M+1}}{Y},  \hspace{0.3cm}  0<\xi <1.
\end{align*}} And so, for $M$ sufficiently large we have 
{\small
\begin{align*}
\left|\frac{x}{\gamma_k}+1\right|^{-1}&<\left(\frac{\gamma_M}{\gamma_{M-1}}-1\right)^{-1}=\frac{\gamma_{M-1}}{\gamma_{M}-\gamma_{M-1}}\\
&\leq C_1(M-1)^{p-\beta}\left(B(M-1)^{\beta}+O((M-1)^{\beta -1})\right)\\
&= C_1(M-1)^p\left(B+O\left(\frac{1}{M-1}\right)\right)\\
&= C_1(M-1)^p B\leq C_1(N-1)^p B  <  C_1N^p B, \hspace{0.2cm} k=1, 2, \ldots , M-1.
\end{align*}}
{\small
\begin{align*}
\left|\frac{x}{\gamma_k}+1\right|^{-1}&<\left(1-\frac{\gamma_{M+1}}{\gamma_{M+2}}\right)^{-1}=\frac{\gamma_{M+2}}{\gamma_{M+2}-\gamma_{M+1}}\\
&\leq C_1(M+1)^{p-\beta}\left(B(M+2)^{\beta}+O((M+2)^{\beta -1})\right)\\
&=C_1(M+1)^p\left(1+O\left(\frac{1}{M+1}\right)\right)\left(B+O\left(\frac{1}{M+2}\right)\right)\\
&=C_1(M+1)^pB\leq C_1N^pB, \hspace{0.2cm} k=M+2, M+3, \ldots \hspace{0.5cm} 1\leq M < N, 
\end{align*}} then for all $\zeta  \in \Gamma^{*}_{2}$ we get
{\small
\begin{align*}
|f(\zeta)|&\leq \frac{2}{a^{2(1-\xi)}_{n}}C_1\cdot BN^p S +\frac{c_M+c_{M+1}}{Y}+1\\
&\leq \frac{2}{a^{2(1-\xi)}_{n}}C_1\cdot BN^p\cdot S+\frac{C_{N+1}}{Y}+1,\hspace{0.3cm} C_{N+1}=\sup_{1\leq k\leq N+1} c_k.
\end{align*}} Therefore, for all  $\zeta  \in \Gamma^{*}_{2}$, inequality $|f(\zeta)|<|g(\zeta)|$  holds if  
\begin{align*}
 \frac{Y^2}{a^{2}_{n}}> \frac{2}{a^{2(1-\xi)}_{n}}C_1\cdot BN^p\cdot S+\frac{C_{N+1}}{Y}+1
\end{align*} or what is the same if 
\begin{align*}
 Y^3-\left(\frac{2}{a^{2(1-\xi)}_{n}}C_1\cdot BN^p\cdot S+1\right)a^{2}_{n}Y-2a^{2}_{n}C_{N+1} >0.
\end{align*} Let $Y>1$, then 
{\small
\begin{align*}
 Y^3-\left(\frac{2C_1\cdot BN^p\cdot S}{a^{2(1-\xi)}_{n}}+1\right)a^{2}_{n}Y-2C_{N+1}a^{2}_{n}>Y^3-\left(\frac{2C_1\cdot BN^p\cdot S}{a^{2(1-\xi)}_{n}}+2C_{N+1}+1\right)a^{2}_{n}Y> 0
\end{align*}} holds, if we choose $Y>a_n\sqrt{\left(\frac{2C_1\cdot BN^p\cdot S}{a^{2(1-\xi)}_{n}}\right)+2C_{N+1}+1}$. Hence, $|f(\zeta)|<|g(\zeta)|$ holds for all  $\zeta  \in \Gamma^{*}$ if  
 \begin{align*}
Y>a_n\cdot \max \left(\sqrt{\left(\frac{2C_1\cdot BN^p\cdot S}{a^{2(1-\xi)}_{n}}\right)+2C_{N+1}+1}, \sqrt{\frac{S}{a^{2(1-\xi)}_{n}}+1}\right).
\end{align*} In  case of  $\zeta  \in \bar{\Gamma}^{*}$, the inequality  $|f(\zeta)|<|g(\zeta)|$ holds under the same condition since $\overline{f(\zeta)}=f(\bar{\zeta})$ and  $\overline{g(\zeta})=g(\bar{\zeta})$. The rest is followed by the proof of Theorem \ref{teorema-de-ceros-numerables}. 

To finish the proof of Theorem  \ref{teorema-de-ceros-infinitos} the basic idea followed  is  to replace the function 
\begin{align*}
\hat{K}(\zeta)= \sum_{k=1}^{\infty} \frac{c_k}{\zeta +\gamma_k} 
\end{align*} by its approximation expressed by the integral
\begin{align*}
h(\zeta)= \int_{1}^{\infty} \frac{A}{t^{\alpha}(\zeta +B t^{\beta})}dt, 
\end{align*} and then to estimate the function $h(\zeta)$ in the region $\{\zeta \in
\mathbb{C}: |\arg \zeta|< \pi-\delta, \delta>0\}$. The approximation of the function 
$\hat{K}(\zeta)$ by $h(\zeta)$ is established in Lemma \ref{lema-sobre-aproximacion}.
\begin{lemma}\label{lema-sobre-aproximacion} If $|\arg(\zeta)|<\pi-\delta$, $\delta>0$, then
\begin{align*}
|\hat{K}(\zeta)-h(\zeta)| \lesssim \frac{1}{|\zeta|}
\end{align*}
\end{lemma}
\begin{proof} This assertion was proved in the paper \cite{VR}. 
\end{proof}

 Let us prove that complex zeroes of the function $l_{n}(\zeta)$ can be asympotically represented in the form $\lambda^{\pm}_{n}=\pm i a_n +\tau_n a_n$, $n \in \mathbb{N}$ as  $n \to \infty$, where $\tau_n \in \mathbb{R}$ is a numerical sequence. For this purpose, it suffices to show that the asymptotic representation $\lambda^{+}_{n}$ satisfies the equation
\begin{align*}
\frac{\hat{K}(\lambda^{+}_{n})}{a^{2(1-\xi)}_{n}}=\frac{\lambda^{+}_{n}}{a^{2}_{n}}+1,
\end{align*}which is equivalent to the equation
\begin{align*}
\hat{K}(\lambda^{+}_{n})= a^{2(1-\xi)}_{n} \tau_n(\tau_n+2i), \hspace{0.3cm}\lambda^{+}_{n}=i a_n +\tau_n a_n.
\end{align*}Hence, we obtain  
\begin{align}\label{tau-igual-a-suma}
\tau_n= \frac{\hat{K}(\lambda^{+}_{n})}{a^{2(1-\xi)}_{n}(\tau_n+2i)}.
\end{align} Denote  $g_n(\tau)$ by
\begin{align}
g_n(\tau)= \frac{\hat{K}(i a_n +\tau a_n)}{a^{2(1-\xi)}_{n}(\tau+2i)}, \hspace{0.3cm}\zeta_n=  i a_n +\tau a_n. 
\end{align} Then the equation (\ref{tau-igual-a-suma}) can be rewritten in the form 
$\tau_n=g_n(\tau_n)$. Hence, it can be concluded that  $\tau_n$ is a fixed point of the map $\tau \to g_n(\tau)$ for $n \to \infty$. Therefore, it suffices to prove that for $n \to \infty$ the map  $\tau \to g_n(\tau)$ is a contraction. Thus, the desired solution $\tau_n \in \mathbb{R}$ will be found as the limit of the sequence $\tau^{k}_{n}$ as $k \to \infty$, where $\tau^{k}_{n}=g(\tau^{k-1}_{n})$, $\tau^{0}_{n}=0, \tau^{1}_{n}\not=\tau_{n}$.

Let us prove that the mapping $\tau \to g_n(\tau)$ is a contraction as $n \to \infty$. This follows from the inequality
\begin{align*}
|g'_n(\tau)|=\left|\frac{\hat{K}(\lambda^{+}_{n})-{\hat{K}}'(\lambda^{+}_{n})(a_n\tau+2ia_n)}{a^{2(1-\xi)}_{n}(\tau+2i)^{2}}\right|\leq \frac{1}{a^{2(1-\xi)}_{n}}\left(|{\hat{K}}'(\lambda^{+}_{n})(2\zeta_n)|+|\hat{K}(\lambda^{+}_{n})|\right)
\end{align*} with $|\tau|<\frac{1}{2}$ and Lemma \ref{lema-sobre-convergencia-suma}.
\begin{lemma} \label{lema-sobre-convergencia-suma} On the region $\{\zeta \in \mathbb{C}: |\arg \zeta|< \pi-\delta, \delta>0\}$ the following relations hold: 
$|\zeta {\hat{K}}'(\zeta)|\to 0$ and  $|\hat{K}(\zeta)|\to 0$ as $|\zeta| \to \infty$. 
\end{lemma}
\begin{proof}The proof of this Lemma was displayed in \cite{VR}. But here is important to mention that
\begin{align*}
g_n(\tau_n)=-\frac{i}{2 }\frac{\hat{K}(ia_n)}{a^{2(1-\xi)}_{n}}[1+O(\tau_n)], \hbox{ $n \to \infty$}.
\end{align*} Indeed, according to Taylor series at  $\tau_n=0$, we get
{\small
\begin{align*}
\frac{\hat{K}(ia_n+\tau_na_n)}{(\tau_n+2i)}=\frac{\hat{K}(ia_n)}{2i}-\frac{\tau_n}{4}({\hat{K}}'(ia_n)(2ia_n)-\hat{K}(ia_{n}))+O(\tau^{2}_n) =-\frac{i}{2}\hat{K}(ia_n)[1+O(\tau_n)].
\end{align*}} Thus, for the sequence $\tau_n \in \mathbb{R}$ the following asymptotic formula holds:
\begin{align}\label{formula-simpotica-de-tau}
\tau_n= -\frac{i}{2 }\frac{\hat{K}(ia_n)}{a^{2(1-\xi)}_{n}}[1+O(\tau_n)], \mbox{ $n \to \infty$}.
\end{align} Further, in order to obtain asymptotic formulas (\ref{aimptotas-de-ceros-infinitos-1})$-$(\ref{aimptotas-de-ceros-infinitos-3}) we use the following Lemma \ref{lema-sobre-convergencia-integral}. 
\end{proof}
\begin{lemma}\label{lema-sobre-convergencia-integral} If $|\arg \zeta|< \pi-\delta, \delta>0$, then 
\begin{align}
&\hat{K}(\zeta)=\frac{AB^{r-1}}{\beta|\zeta|^{r}}\int_{0}^{\infty}\frac{dt}{t^{r}(\exp(i\varphi)+t)}+O\left(\frac{1}{|\zeta|}\right), &\mbox{for $0<r<1$},\\
&\hat{K}(\zeta)=\frac{A}{\beta}\frac{\ln\left|\frac{\zeta}{B}+1\right|}{\zeta}+O\left(\frac{1}{|\zeta|}\right),  &\mbox{for $r=1$}.
\end{align}
\end{lemma}
\begin{proof} The proof of this Lemma was displayed in \cite{VR}.
\end{proof}
 And so, the following asymptotic formula, in the case $0<r<1$,
{\small
\begin{align*}
-\frac{i}{2}\hat{K}(ia_n)=-\frac{i}{2}\frac{AB^{r-1}}{\beta|a_n|^{r}}\int_{0}^{\infty}\frac{dt}{t^{r}(i+t)}+O\left(\frac{1}{|a_n|}\right)=-\frac{AD}{\beta B^{1-r}}\frac{1}{a^{r}_n} +O\left(\frac{1}{a_n}\right),
\end{align*}}  can be obtained  using the Lemma  \ref{lema-sobre-convergencia-integral} and result
{\small
\begin{align*}
D=\frac{i}{2}\int_{0}^{\infty}\frac{dt}{t^{r}(i+t))}=\frac{1}{2}\frac{\pi}{\sin(\pi r)}\exp(i\frac{\pi}{2}(1-r)), \hspace{0.4cm}\mbox{$\arg(ia_n)=\frac{\pi}{2}=\varphi$}.
\end{align*}} That is to say
{\small
\begin{align*}
\tau_n&=  -\frac{i}{2 }\frac{\hat{K}(ia_n)}{a^{2(1-\xi)}_{n}}[1+O(\tau_n)]=\left(-\frac{AD}{\beta B^{1-r}}\frac{1}{a^{r+2(1-\xi)}_n} +O\left(\frac{1}{a^{1+2(1-\xi)}_n}\right)\right)[1+O(\tau_n)].
\end{align*}}
Set $M:=-\frac{AD}{\beta B^{1-r}}$. Then 
{\small
\begin{align*}
\tau_n & = \frac{M}{a^{r+2(1-\xi)}_n}+O\left(\frac{1}{a^{1+2(1-\xi)}_n}\right)+\\
&+\left(\frac{M}{a^{r+2(1-\xi)}_n}+O\left(\frac{1}{a^{1+2(1-\xi)}_n}\right)\right) O\left(\frac{ M }{a^{r+2(1-\xi)}_n}+O\left(\frac{1}{a^{1+2(1-\xi)}_n}\right)\right)=\\
&=\frac{M}{a^{r+2(1-\xi)}_n}+O\left(\frac{1}{a^{1+2(1-\xi)}_n}\right)+O\left(\frac{1}{a^{r+1+2(1-\xi)}_n}\right)+O\left(\frac{1}{a^{2r+2(1-\xi)}_n}\right)+O\left(\frac{1}{a^{2}_n}\right).
\end{align*}}For the case $0<r<1$, we obtain the asymptotic formula 
\begin{align}
\tau_n & = \frac{M}{a^{r+2(1-\xi)}_n}+O\left(\frac{1}{a^{2r+2(1-\xi)}_n}\right), \hspace{0.3cm}\mbox{ $0 < r < \frac{1}{2}$},\label{asimpotota-tau-1}\\
\tau_n &=\frac{M}{a^{r+2(1-\xi)}_n}+O\left(\frac{1}{a^{1+2(1-\xi)}_n}\right),\hspace{0.3cm}\mbox{ $\frac{1}{2} \leq r<1$}.\label{asimpotota-tau-2}
\end{align} For the case $r=1$, using again the Lemma \ref{lema-sobre-convergencia-integral}, we have
{\small
\begin{align*}
-\frac{i}{2}\hat{K}(ia_n)&=-\frac{i}{2}\frac{A}{\beta}\frac{\ln\left|\frac{ia}{B}+1\right|}{ia_n}+O\left(\frac{1}{a_n}\right)=-\frac{1}{2}\frac{A}{\beta}\frac{\ln\sqrt{\frac{a^{2}_n}{B^2}+1}}{a_n}+O\left(\frac{1}{a_n}\right)\\
&=-\frac{1}{2}\frac{A}{\beta}\frac{\ln a_n}{a_n}+O\left(\frac{1}{a_n}\right). 
\end{align*}}  That is, we obtain another asymptotic formula, 
{\small
\begin{align}\label{asimpotota-tau-3}
\tau_n&= -\frac{i}{2}\frac{\hat{K}(ia_n)}{a^{2(1-\xi)}_n}[1+O(\tau_n)]\nonumber\\
&=-\frac{1}{2}\frac{A}{\beta}\frac{\ln a_n}{a^{1+2(1-\xi)}_n}+O\left(\frac{1}{a^{1+2(1-\xi)}_n}\right)+\nonumber \\
&+\left(-\frac{1}{2}\frac{A}{\beta}\frac{\ln a_n}{a^{1+2(1-\xi)}_n}+O\left(\frac{1}{a^{1+2(1-\xi)}_n}\right)\right) O\left( -\frac{1}{2}\frac{A}{\beta}\frac{\ln a_n }{a^{1+2(1-\xi)}_n}+O\left(\frac{1}{a^{1+2(1-\xi)}_n}\right)\right)\nonumber\\
&=-\frac{1}{2}\frac{A}{\beta}\frac{\ln a_n}{a^{1+2(1-\xi)}_n}+O\left(\frac{1}{a^{1+2(1-\xi)}_n}\right). 
\end{align}} We recall that 
\begin{align}
\lambda^{\pm}_{n}=\im \lambda^{\pm}_{n}+\re \lambda^{\pm}_{n}=\pm ia_n+\tau_n a_n 
\end{align} and so from asymptotic formulas (\ref{asimpotota-tau-1}) $-$ (\ref{asimpotota-tau-3}), we have 
\begin{align}
\re \lambda^{\pm}_n (\xi)& = \frac{M}{a^{r+1-2\xi}_n}+O\left(\frac{1}{a^{2(r-\xi)+1}_n}\right),\hspace{0.3cm}&\mbox{ $0<r<\frac{1}{2}$},\label{asympota-real1}\\
\re \lambda^{\pm}_n (\xi)&=\frac{M}{a^{r+1-2\xi}_n}+O\left(\frac{1}{a^{2(1-\xi)}_n}\right),\hspace{0.3cm}&\mbox{ $\frac{1}{2}\leq r<1$},\label{asympota-real2}\\
\re \lambda^{\pm}_n (\xi) &= -\frac{1}{2}\frac{A}{\beta}\frac{\ln a_n}{a^{2(1-\xi)}_n}+O\left(\frac{1}{a^{2(1-\xi)}_n}\right), \hspace{0.3cm}&\mbox{ $r=1$}.
\end{align} 
\end{proof}

We consider the dependence of the asymptotic behavior of the complex roots as from the properties of the kernel of the equation
 $\hat{K}$ and parameter $\xi$. If we consider the case  $ \xi \in (0, 1)$, for $a_n \to \infty$, then the complex roots tend to the imaginary axis. Indeed, 

\begin{enumerate} [\ i)]
\item For the case $r=1$, we get $\re \lambda^{\pm}_n (\xi) \to 0$ for all  $ \xi \in (0, 1)$. 
\item \label{the-rest-case}The rest of the case, i.e. when $r \in (0, 1)$,  we have $\re \lambda^{\pm}_n (\xi) \to 0$ if and only if  $r+1-2\xi>0$, that is, $\xi < \frac{r+1}{2}$ for $a_n \to \infty$. 
\end{enumerate}

\section{Remarks and Comments}\label{remarks-and-comments} 
In this paper we analyzed asymptotic behavior of the spectrum of the operator-valued function $L(\zeta)$, first for the case $\sum_{k=1}^{\infty} c_k< +\infty$, and the second when condition $\sum_{k=1}^{\infty} c_k< +\infty$ does not satisfies. For the first case, we obtained the asymptotic formulas  (\ref{asymptota-sum-finita-para-mu1}), (\ref{asymptota-sum-finita-para-mu2}) and (\ref{asymptota-sum-finita-para-mu3}) for different values of  $\xi$ where $\re \mu^{\pm}_n (\xi)\to 0$ as $a_n \to 0$. That is, in our case, the solutions of system (\ref{sytem1.1})-(\ref{sytem1.2}) does not decay exponentially for all $\xi \in (0, 1)$ unlike \cite{JMF} and \cite{MFBL}.

In the second case,  for the sequences $\{c_{k}\}^{\infty}_{k=1}$, $\{\gamma_{k}\}^{\infty}_{k=1}$, was  considered the  asymptotic representation 
{\small
\begin{align*}
   c_k=\frac{A}{k^{\alpha}} + O\left( \frac{1}{k^{\alpha +1}}\right), \hspace{0.2cm}  \gamma_k=B k^{\beta} + O( k^{\beta-1}), \hspace{0.2cm}  k \in \mathbb{N}
\end{align*}}together with the following additional conditions (\ref{gama-k+1-menos-gama-k-1}) and (\ref{gama-k+1-menos-gama-k-2}). From the asymptotic formulas (\ref{aimptotas-de-ceros-infinitos-1}), (\ref{aimptotas-de-ceros-infinitos-2}) and (\ref{aimptotas-de-ceros-infinitos-3}) we obtained the following analysis: considering the dependence of the asymptotic behavior of the complex zeroes as from the properties of the kernel of the equation $\hat{K}$ and parameter $\xi$ we obtained that  the complex zeroes tend to the imaginary axis for all $\xi \in (0, 1)$. Indeed, for $r=1$, we get $\re \lambda^{\pm}_n (\xi)\to 0$ for all $\xi \in (0, 1)$. Here, the solutions for system (\ref{sytem1.1})-(\ref{sytem1.2}) again does not decay exponentially for all $\xi \in (0, 1)$.  For the  rest  case, i.e. for $r \in (0, 1)$, the solutions for the system (\ref{sytem1.1})-(\ref{sytem1.2}) does not decay exponentially only if  $\xi \in (0, \frac{r+1}{2})$ because
 $\re \lambda^{\pm}_n (\xi) \to 0$ if and only if $\xi \in (0, \frac{r+1}{2})$.

We get a different case when we consider $r+1-2\xi<0$, i.e., if $\xi \in (\frac{r+1}{2}, 1)$ the solutions for system (\ref{sytem1.1})-(\ref{sytem1.2}) decay exponentially. Indeed for all $r, \xi$ such that $r \in (0, 1)$ and $\xi \in (\frac{r+1}{2}, 1)$  we have
\begin{align*}
\re \lambda^{\pm}_n (\xi)& = M a^{(r+1-2\xi)}_n+O\left(a^{(2 \xi -1)}_{n}\right),\hspace{0.2cm}&\mbox{$0<\xi<\frac{1}{2}$},\\
\re \lambda^{\pm}_n (\xi) &= M a^{(r+1-2\xi)}_n+O\left(\frac{1}{a^{(2\xi-1)}_n}\right), \hspace{0.2cm}&\mbox{$\frac{1}{2}<\xi<1$},\\
\re \lambda^{\pm}_n (\xi)&= M a^{r}_n+O\left(1\right),\hspace{0.2cm}&\mbox{$\xi=\frac{1}{2}$},  
\end{align*} where $M:=-\frac{AD}{\beta B^{1-r}}$. Here,  $\re \lambda^{\pm}_n (\xi) \to - \infty$ as $a_n \to \infty$.

If $\xi=\frac{r+1}{2}$, then the asymptotic formulas (\ref{asympota-real1}) and (\ref{asympota-real2}) depend only on $r$. Indeed, for all $r \in (0, 1)$, we have
\begin{align*}
\re \lambda^{\pm}_n & = -\frac{AD}{\beta B^{1-r}}+O\left(\frac{1}{a^{r}_{n}}\right)\\
\re \lambda^{\pm}_n &= -\frac{AD}{\beta B^{1-r}} +O\left(\frac{1}{a^{(1-r)}_n}\right). 
\end{align*} 

\end{document}